\documentclass{amsart}

\usepackage{amsmath,amsthm,mathrsfs,amssymb,graphicx,comment}
\usepackage[all]{xy}
\usepackage{tikz}
\usetikzlibrary{shapes,snakes}
\usepackage[pdfborder={0 0 0},backref=false,colorlinks=true,linktocpage=true]{hyperref}

\newtheorem{theorem}{Theorem}[section]
\newtheorem{proposition}[theorem]{Proposition}
\newtheorem{lemma}[theorem]{Lemma}
\newtheorem{corollary}[theorem]{Corollary}
\theoremstyle{definition}
\newtheorem{definition}[theorem]{Definition}

\begin{document}

\title{The Complexity of Primes in Computable UFDs}

\author[Dzhafarov]{Damir D. Dzhafarov}
\address{Department of Mathematics\\
Department of Mathematics\\
University of Connecticut\\
Storrs, Connecticut 06269
U.S.A.}
\email{damir@math.uconn.edu}

\author[Mileti]{Joseph R. Mileti}
\address{Department of Mathematics and Statistics\\
Grinnell College\\
Grinnell, Iowa 50112
U.S.A.}
\email{miletijo@grinnell.edu}

\thanks{Dzhafarov was partially supported by an NSF Postdoctoral Fellowship.  We thank Sean Sather-Wagstaff for pointing out Nagata'a Criterion to us, and Keith Conrad for helpful discussions.}

\begin{abstract}
In many simple integral domains, such as $\mathbb{Z}$ or $\mathbb{Z}[i]$, there is a straightforward procedure to determine if an element is prime by simply reducing to a direct check of finitely many potential divisors.  Despite the fact that such a naive approach does not immediately translate to integral domains like $\mathbb{Z}[x]$ or the ring of integers in an algebraic number field, there still exist computational procedures that work to determine the prime elements in these cases.  In contrast, we will show how to computably extend $\mathbb{Z}$ in such a way that we can control the ordinary integer primes in any $\Pi_2^0$ way, all while maintaining unique factorization.  As a corollary, we establish the existence of a computable UFD such that the set of primes is $\Pi_2^0$-complete in every computable presentation.
\end{abstract}

\maketitle

\section{Introduction}

The power and versatility of modern algebra arise from the abstract and axiomatic approach it takes.  However, with the rise of computer algebra systems, it is important to find algorithms in order to perform computations within these algebraic structures.  Of course, in these settings, one also cares about the efficiency of these procedures.  For example, although the primes in $\mathbb{Z}$ are trivially computable, there is a great deal of interest in how quickly we can determine whether an element is prime.  In contrast, it is known that there are computable integral domains where it is impossible even in principle to determine the primes computationally (see below).  In this paper, we extend these examples to build a computable UFD where the primes are maximally complicated in a very strong sense.  We begin with the following definition (see \cite{Soare} for background on the formal definitions of computable sets and function).

\begin{definition}
A {\em computable ring} is a ring whose underlying set is a computable set $A \subseteq \mathbb{N}$, with the property that $+$ and $\cdot$ are computable functions from $A \times A$ to $A$.
\end{definition}

For example, it is easy to view $\mathbb{Z}$ as a computable ring by using the even natural numbers to code the positive elements in ascending order and the odd natural numbers to code the negative elements in descending order.  Of course, we can view $\mathbb{Z}$ as a computable ring in a different way by switching the roles of the evens and odds.  Thus, a given ring can have multiple distinct computable {\em presentations}.  Many other natural rings can also be viewed as computable rings.  Since we can code relatively prime pairs of natural numbers using a single natural number, we can view $\mathbb{Q}$ as a computable ring.  Similarly, since we can code finite sequences of integers as natural numbers, we can view $\mathbb{Z}[x]$ as a computable ring as well.  Generalizing this, given an arbitrary computable ring $A$, we can realize the polynomial ring $A[x]$ as a computable ring in a natural way.  In contrast, uncountable rings can never be viewed as computable rings, and there are some countable rings that can not as well.

For a general overview of results about computable rings and fields, see \cite{SHTucker}.  Computable fields have received a great deal of attention (\cite{FrohlichShep}, \cite{MetakidesNerode}, \cite{Rabin}), and \cite{MillerNotices} provides an excellent overview of work in this area.  For computable rings, several papers (\cite{ConidisRadical}, \cite{Ideals}, \cite{FSS}, \cite{FSSAdd}) have studied the complexity of ideals and radicals  from the perspective of computability theory and reverse mathematics.

The following algebraic definitions are standard.

\begin{definition}
Let $A$ be an integral domain, i.e.~a commutative ring with $1 \neq 0$ and with no zero divisors (so $ab = 0$ implies either $a = 0$ or $b=0$).  Recall the following definitions.
\begin{enumerate}
\item An element $u \in A$ is a {\em unit} if there exists $w \in A$ with $uw = 1$.  We denote the set of units by $U(A)$.  Notice that $U(A)$ is a multiplicative group.
\item Given $a,b \in A$, we say that $a$ and $b$ are {\em associates} if there exists $u \in U(A)$ with $au = b$.  We denote the set of associates of $a$ by $Associates_A(a)$.
\item An element $p \in A$ is {\em irreducible} if it nonzero, not a unit, and has the property that whenever $p = ab$, either $a$ is a unit or $b$ is a unit.  An equivalent definition is that $p \in A$ is irreducible if it is nonzero, not a unit, and its divisors are precisely the units and the associates of $p$.
\item An element $p \in A$ is {\em prime} if it nonzero, not a unit, and has the property that whenever $p \mid ab$, either $p \mid a$ or $p \mid b$.  We denotes the set of primes of $A$ by $Primes(A)$.
\item $A$ is a {\em unique factorization domain}, or {\em UFD}, if it has the following two properties:
\begin{itemize}
\item For each $a \in A$ such that $a$ is nonzero and not a unit, there exist irreducible elements $r_1,r_2,\dots,r_n \in A$ with $a = r_1r_2 \cdots r_n$.
\item If $r_1,r_2,\dots,r_n,q_1,q_2,\dots,q_m \in A$ are all irreducible and $r_1r_2 \cdots r_n = q_1q_2 \cdots q_m$, then $n = m$ and there exists a permutation $\sigma$ of $\{1,2,\dots,n\}$ such that $r_i$ and $q_{\sigma(i)}$ are associates for all $i$.
\end{itemize}
\end{enumerate}
\end{definition}

It is a simple fact that if $A$ is an integral domain, then every prime element of $A$ is irreducible.  The converse fails in general, but is true in every UFD.  In fact, we have the following standard result.

\begin{theorem} \label{t:EquivCharOfUFD}
Let $A$ be an integral domain.  The following are equivalent:
\begin{enumerate}
\item $A$ is a UFD.
\item Every element of $A$ that is nonzero and not a unit is a product of irreducibles, and every irreducible element of $A$ is prime.
\end{enumerate}
\end{theorem}

Of course, for most computable integral domains that arise in practice, the set of primes form a computable set in any natural computable presentation.  For the ring $\mathbb{Z}$, the set of primes trivially form a computable set.  Kronecker showed that the set of primes in (any reasonable computable presentation of) the UFD $\mathbb{Z}[x]$ is computable.  Using Gauss' Lemma and the fact that every element of $\mathbb{Q}[x]$ is an associate of an element of $\mathbb{Z}[x]$, it follows that the set of primes in $\mathbb{Q}[x]$ is computable as well. 

Consider a number field $K$ with $[K : \mathbb{Q}] = n$ and let $\mathcal{O}_K$ be the set of algebraic integers in $K$.  In general, $\mathcal{O}_K$ is always a Dedekind domain, but it may not be a UFD.  We may fix an integral basis of $K$ over $\mathbb{Q}$, i.e.~fix $b_1,b_2,\dots,b_n \in \mathcal{O}_K$ that form a basis for $K$ over $\mathbb{Q}$ such that
\[
\mathcal{O}_K = \{m_1b_1 + m_2b_2 + \dots + m_nb_n : m_i \in \mathbb{Z}\}.
\]
Now given the finitely many values $b_i \cdot b_j$, we can compute the multiplication function on $K$ and hence on $\mathcal{O}_K$ as well.  Since we can simply hard code in these values, it follows that any integral basis provides a computable presentation of the field $K$ (by working with underlying set $\mathbb{Q}^n$) and the ring $\mathcal{O}_K$ (by working with underlying set $\mathbb{Z}^n$).  We have the following fact.

\begin{proposition}
Let $K$ be a number field with $[K : \mathbb{Q}] = n$.  If we fix an integral basis of $K$ over $\mathbb{Q}$, and represent elements of $\mathcal{O}_K$ using elements of $\mathbb{Z}^n$, then the set of primes elements of $\mathcal{O}_K$ is computable.
\end{proposition}

\begin{proof}
Given $\alpha \in K$, the map $\varphi_{\alpha} \colon K \to K$ defined by $\varphi_{\alpha}(x) = \alpha \cdot x$ is a $\mathbb{Q}$-linear map, and moreover we can uniformly compute a matrix $M_{\alpha}$ with rational entries representing this map because we need only express $\alpha \cdot b_i$ in terms of our basis.  Furthermore, notice that if $\alpha \in \mathcal{O}_K$, then $\varphi_{\alpha}$ maps $\mathcal{O}_K$ into $\mathcal{O}_K$, and hence $M_{\alpha}$ has integer entries.  From this, we can conclude that the norm map $N \colon \mathcal{O}_K \to \mathbb{Z}$ defined by $N(\alpha) = \det(\varphi_{\alpha}) = \det(M_{\alpha})$ is a computable function.  Since an element $\alpha \in \mathcal{O}_K$ is a unit if and only if $N(\alpha) = \pm 1$, it follows that $U(\mathcal{O}_K)$ is a computable set.

Moreover, given $\alpha,\beta \in K$ with $\alpha \neq 0$ represented as elements of $\mathbb{Q}^n$, we can uniformly compute $\frac{\beta}{\alpha}$ as represented by an element of $\mathbb{Q}^n$ by simply searching through the effectively countable set $\mathbb{Q}^n$ until we find $\gamma \in K$ with $\gamma \cdot \alpha = \beta$.  Now if $\alpha,\beta \in \mathcal{O}_K$, we can effectively determine if $\alpha \mid \beta$ in $\mathcal{O}_K$ by checking if this representation of $\frac{\beta}{\alpha}$ is in $\mathbb{Z}^n$.  Therefore, the divisibility relation on $\mathcal{O}_K$ is computable.

Since we can compute the norm of an element, and since $|\mathcal{O}_K / \langle \alpha \rangle| = |N(\alpha)|$, we can compute the function $f \colon \mathcal{O}_K \backslash \{0\} \to \mathbb{N}$ defined by $f(\alpha) = |\mathcal{O}_K / \langle \alpha \rangle|$.  Now to determine if $\alpha$ is prime, we compute $f(\alpha)$ and then search until we find $f(\alpha)$ many distinct representative of the quotient (this is possible because the divisibility relation is computable).  With these representatives, we can form the finite multiplicative table of the quotient (again using the fact that the divisibility relation is computable).  To determine if $\alpha$ is prime, we then check if the quotient has any zero divisors, which is now just a finite check.
\end{proof}

Despite all of this, there are computable integral domains such that the set of primes is not computable.  In fact, there is a computable field $F$ such that the set of primes in $F[x]$ is not computable (see \cite[Lemma 3.4]{MillerNotices} or \cite[Section 3.2]{SHTucker} for an example).  There exist methods to measure the complexity of sets that are not computable, and we investigate the placement of such sets in the arithmetical hierarchy arising from quantifying over computable relations.

\begin{definition}
Let $Z \subseteq \mathbb{N}$.
\begin{itemize}
\item We say that $Z$ is a $\Sigma_1^0$ set if there exists a computable $R \subseteq \mathbb{N}^2$ such that
\[
i \in Z \Longleftrightarrow (\exists x) R(x,i).
\]
\item We say that $Z$ is a $\Pi_1^0$ set if there exists a computable $R \subseteq \mathbb{N}^2$ such that
\[
i \in Z \Longleftrightarrow (\forall x) R(x,i).
\]
\item We say that $Z$ is a $\Pi_2^0$ set if there exists a computable $R \subseteq \mathbb{N}^3$ such that
\[
i \in Z \Longleftrightarrow (\forall x)(\exists y) R(x,y,i).
\]
\end{itemize}
\end{definition}

Since it is possible to computably code finite sequences of natural numbers with a single natural number, the above definitions do not change if we allow finite consecutive blocks of the same (existential or universal) quantifiers.  Although every computable set is $\Sigma_1^0$, there exists a $\Sigma_1^0$ set that is not computable, such as the set of natural numbers coding programs that halt.  Similarly, the collection of $\Sigma_1^0$ sets is a proper subset of the collection of all $\Pi_2^0$ sets, and the collection of $\Pi_1^0$ sets is a proper subset of the collection of all $\Pi_2^0$ sets.  See \cite[Chapter 4]{Soare} for more information about the arithmetical hierarchy.

Suppose that $A$ is a computable integral domain.  We then have that $U(A)$ is a $\Sigma_1^0$ set because
\[
u \in U(A) \Longleftrightarrow (\exists w)[uw = 1],
\]
and the relation $uw = 1$ is computable.  The set of irreducibles of $A$ is a $\Pi_2^0$ set because $p$ is irreducible in $A$ if and only if
\[
p \neq 0 \wedge (\forall c)[pc \neq 1] \wedge (\forall a)(\forall b)[p = ab \rightarrow a \in U(A) \vee b \in U(A)],
\]
and we already know that $U(A)$ is a $\Sigma_1^0$ set.  A similar analysis shows that the set of primes of $A$ is a $\Pi_2^0$ set.  Our main result is the following, which says that this result is best possible in a very strong sense.

\begin{theorem} \label{t:Pi2ControlOfPrimes}
Let $Q$ be a $\Pi_2^0$ set, and let $p_0,p_1,p_2,\dots$ list the usual primes from $\mathbb{N}$ in increasing order.  There exists a computable UFD $A$ such that:
\begin{itemize}
\item $\mathbb{Z}$ is a subring of $A$.
\item $p_i$ is prime in $A$ if and only if $i \in Q$.
\end{itemize}
\end{theorem}

This theorem differs from the result that there is a computable field $F$ such that the set of prime elements in $F[x]$ is not computable.  One reason is that we are working directly with the usual primes rather than coding into polynomials (such as $x^2 - p_i$), or creating our own primes to do the coding.  As a result, our approach has a more number-theoretic flavor.  Furthermore, if $F$ is a computable field, then $U(F[x])$ is a computable set in any reasonable computable presentation of $F[x]$, so the set of irreducibles (and hence primes) of $F[x]$ will always be a $\Pi_1^0$ set by our above analysis, and hence could not be $\Pi_2^0$-complete.  Moreover, we obtain the following strong corollary that may not hold if we code complexity into other primes.

\begin{corollary}
There exists a computable UFD $A$ such that the set of primes of $A$ is $\Pi_2^0$-complete in every computable presentation of $A$, uniformly in an index for the presentation.
\end{corollary}

\begin{proof}
Fix a $\Pi^0_2$-complete set $Q$ (see \cite{Soare}, Theorem IV.3.2), and construct $A$ using this $Q$ as in Theorem \ref{t:Pi2ControlOfPrimes}.  Now given any computable presentation of $A$, we can find the multiplicative identity element of $A$ by searching until we find $a \in A$ such that $a^2 = a$ and $a+a \neq a$ (notice that the multiplicative identity is the only such element because $A$ is an integral domain).  With this element in hand, we can find the representation of each $p_i$ in $A$ by adding the multiplicative identity to itself the required number of times.  Therefore, the set of primes of $A$ is $\Pi_2^0$-complete in every computable presentation of $A$.
\end{proof}

Since we are working with the normal integer primes rather than creating some new ones, we need to be much more careful because of the algebraic dependence relationships that exist between them.  By adjusting the status of one prime, i.e.~keeping it prime or making it not prime, it is certainly conceivable that we could interfere with others.  For example, suppose that $A$ is a integral domain, that $q \in Primes(A)$, and that we want to break the primeness/irreducibility of $q$, i.e.~we want to introduce a nontrivial factorization of $q$.  One idea is to introduce a square root of $q$, i.e.~to introduce a new element $x$ with $x^2 = q$.  The natural way to do this is to consider $A[x] / \langle x^2 - q \rangle$, but this is problematic for a few reasons.  With this approach, we might destroy the primeness/irreducibility of other elements in $A$, as it is well-known that if $p,q \in \mathbb{N}$ are distinct odd primes, then $p$ is not prime in $\mathbb{Z}[\sqrt{q}] \cong \mathbb{Z}[x] / \langle x^2 - q \rangle$ if and only if $q$ is a square modulo $p$.  For example, in $\mathbb{Z}[\sqrt{7}]$, we have that $3$ is not prime because $3 \mid (1-\sqrt{7})(1+\sqrt{7})$ but $3 \nmid 1-\sqrt{7}$ and $3 \nmid 1+\sqrt{7}$.  Moreover, in $\mathbb{Z}[\sqrt{q}]$, irreducibles might fail to be prime, and hence we may have lost the property of being a UFD.  Finally, with this approach it is also impossible to later destroy this factorization as we can not make $x$ a unit without making $q$ a unit.

Another potential issue arises if we do want to destroy a given factorization by making an element a unit, but we are not in a UFD and/or are working with irreducibles.  For example, in \cite{KConrad-Quad}, the following example is given:  in $\mathbb{Z}[\sqrt{-14}]$ one has
\[
3 \cdot 3 \cdot 3 \cdot 3 = (5 + 2\sqrt{-14})(5 - 2\sqrt{-14})
\]
where all of the above factors are irreducible.  It follows that
\[
5 + 2 \sqrt{-14} \mid 3^4
\]
even though $5+2\sqrt{-14}$ and $3$ are not associates in $\mathbb{Z}[\sqrt{-14}]$ (as the units are $\pm 1$).  Thus, in this ring, if we later make $3$ a unit, then we must make $5 + 2 \sqrt{-14}$ a unit as well.

With all of these potential issues in mind, we now outline the idea behind the construction.  Start with $A_0 = \mathbb{Z}$.  We want to turn the normal primes $p_i$ {\em on} and {\em off} based on a $\Pi_2^0$ set $Q$.  Fix a computable $R \subseteq \mathbb{N}^3$ such that
\[
i \in Q \Longleftrightarrow (\forall w)(\exists z) R(w,z,i)
\]
So intuitively if $i$ acts infinitely often (i.e.~if for each $w$ in turn, we find a witnessing $z$), then we want $p_i$ to be prime in the end.  If $i$ acts finitely often, we want $p_i$ not to be prime.  To work for $i$, we assume finite action, and introduce a factorization $p_i = x_iy_i$ for new elements $x_i$ and $y_i$.  If $i$ acts at a later stage, we want to destroy this factorization.  To do this, we make $y_i$ a unit.  We will show that this keeps $x_i$ prime, and since $p_i$ will now be an associate of $x_i$, we will reinstate the fact that $p_i$ is prime.  We then introduce another factorization $p_i = x_i'y_i'$ for new $x_i'$ and $y_i'$, and continue, destroying it if $i$ acts again.  We do this forever, building a chain of integral domains $\mathbb{Z} = A_0 \subseteq A_1 \subseteq A_2 \subseteq \dots$.  Let $A_{\infty} = \bigcup_{n=0}^{\infty} A_n$. 

We build this ring in a computable fashion as follows.  We think of the natural numbers as being split into infinitely many infinite columns through a computable pairing function.  We start by putting the integers in the first column and call that $A_0$.  Now each extension will add infinitely many elements to the ring, and to do this at a given stage we will simply add these elements into the next column and computably define addition and multiplication at this point both within this column and between this column and previous ones.  Eventually, we will fill up all of the columns in turn, and define all of the operations, resulting in a computable ring.

With this construction, we will need to keep track of several things.  For example, when we make an element a unit, we will localize our ring, and since we have already constructed part of the ring so far we will need to ensure that we can computably determine the new elements to add in order to form this localization.  As a result we will need to ensure that we can computably keep track of the multiples of the $x_i$ and $y_i$ that we introduce.  Algebraically, we need to ensure that the rings along the way are all Noetherian UFDs and that unrelated primes are unaffected by these operations.  Finally, we need to check that this limiting ring has the required properties since a union of UFDs need not be a UFD in general.

\section{Turning a Prime into a Unit}

Let $A$ be an integral domain and let $q \in A$ be prime.  Suppose that we want to embed $A$ in another integral domain $B$ such that $q$ is a unit in $B$.  Naturally, one considers the corresponding localization, i.e.~we take the multiplicative set $S = \{1,q,q^2,\dots\}$ and let $B = S^{-1}A$.  Thinking of $A$ as sitting inside its field of fractions $F$ by identifying $a$ with $\frac{a}{1}$, we have
\begin{align*}
B & = \left\{\frac{a}{q^k} : a \in A \text{ and } k \geq 0\right\} \\
& = A \cup \left\{\frac{a}{q^k} : a \in A \text{ and } k \geq 1\right\}.
\end{align*}
Now if $A$ is a computable integral domain and we want to think about extending to $B$ in a computable fashion, then we need to know which of the elements in the set on the right are really new, along with when they are distinct from each other  For example, we have that $\frac{q^2}{q} = \frac{q}{1}$ is already an element of $A$, so we do not want to introduce it.

Notice that every element of $B \backslash A$ can be written in the form $\frac{a}{q^k}$ where $k \geq 1$ and $q \nmid a$.  To see this, suppose that we are given a general $\frac{a}{q^m} \in B$ with $a \in A$ and $m \geq 1$.  If $q \mid a$, we can factor out a $q$ from $a$ and cancel terms to obtain a different representation of the same element with a smaller power of $q$ in the denominator.  We can now induct (or take a minimal power in the denominator) to argue that this element is represented in the above set.  Thus, we have
\[
B = A \cup \left\{\frac{a}{q^k} : a \in A, q \nmid a, \text{ and } k \geq 1\right\}.
\]
Moreover, it is not difficult to show that the above representations are unique (i.e.~that $\frac{a}{q^k} \notin A$ when $q \nmid a$ and $k \geq 1$, and also that two elements of the right set are equal exactly when the numerator and power of $q$ are equal.  For the latter, if $\frac{a}{q^k} = \frac{b}{q^{\ell}}$ with $q \nmid a$ and $q \nmid b$, then $q^{\ell}a = q^kb$.  Cancel common $q$'s.  If $k = \ell$, then $a = b$, and we are done.  Otherwise, we have a $q$ left over on side, so either $q \mid a$ or $q \mid b$, a contradiction).

As a result, if $A$ is computable, and the set $\{a \in A : q \mid a\}$ is computable, then from $A$, $q$, and an index for this set we can uniformly computably build $B$ as an extension of $A$.  Since we are going to repeatedly apply this construction along with a factorization construction, we will need to ensure that the set of multiples of other primes remain computable as well.

\begin{proposition} \label{p:UnitsInLocalization} 
We have
\[
U(B) = U(A) \cup \{uq^k : k \geq 1 \text{ and } u \in U(A)\} \cup \left\{\frac{u}{q^k} : k \geq 1 \text{ and } u \in U(A)\right\}.
\]
\end{proposition}

\begin{proof}
Since $q \cdot \frac{1}{q} = 1$, we have that $q$ and $\frac{1}{q}$ are both elements of $U(B)$.  We trivially have that $U(A) \subseteq U(B)$ and also that $U(B)$ is closed under multiplication.  It follows that every element of the sets on the right is an element of $U(B)$.

We now show the reverse containment.  Let $\sigma \in U(B)$ be arbitrary.  Suppose first that $\sigma = a \in A$.  We have two cases.
\begin{itemize}
\item Suppose that there exists $b \in A$ with $\sigma b = ab = 1$.  We then trivially have that $\sigma = a \in U(A)$.
\item Suppose instead that there exists $b \in A$ with $q \nmid b$ and $\ell \geq 1$ such that $\sigma \cdot \frac{b}{q^{\ell}} = \frac{a}{1} \cdot \frac{b}{q^{\ell}} = 1$.  We then have $q^{\ell} = ab$.  Since $q$ is prime and $q \nmid b$, it follows from Lemma \ref{l:PrimePowerDividesProductButNotTerm} that $q^{\ell} \mid a$.  Fix $c \in A$ with $a = q^{\ell}c$.  We then have $q^{\ell} = ab = q^{\ell}cb$, so $cb = 1$ and hence $c \in U(A)$.  Thus, $\sigma = a = cq^{\ell}$ where $c \in U(A)$.
\end{itemize}
Suppose now that that $\sigma = \frac{a}{q^k}$ where $a \in A$ with $q \nmid a$ and $k \geq 1$.
\begin{itemize}
\item Suppose that there exists $b \in A$ with $\sigma b = \frac{a}{q^k} \cdot \frac{b}{1} = 1$.  We then have $q^k = ab$.  Since $q \nmid a$ and $q$ is prime, we conclude from Lemma \ref{l:PrimePowerDividesProductButNotTerm} that $q^k \mid b$.  Fix $c \in A$ with $b = cq^k$.  We then have $q^k = ab = acq^k$, so $ac = 1$ and hence $a \in U(A)$.  Thus, $\sigma = \frac{a}{q^k}$ with $a \in U(A)$.
\item Suppose instead that there exists $b \in A$ with $q \nmid b$ and $\ell \geq 1$ such that $\sigma \cdot \frac{b}{q^{\ell}} = \frac{a}{q^k} \cdot \frac{b}{q^{\ell}} = 1$.  We then have that $ab = q^{k+\ell}$.  Since $k,\ell \geq 1$, this implies that $q \mid ab$ and hence either $q \mid a$ or $q \mid b$ (since $q$ is prime), a contradiction.
\end{itemize}
This completes the proof.
\end{proof}

\begin{lemma} \label{l:PrimePowerDividesProductButNotTerm}
Let $A$ be an integral domain, let $p \in A$ be prime, and let $k \geq 1$.  If $p^k \mid ab$ and $p \nmid a$, then $p^k \mid b$.
\end{lemma}

\begin{proof}
By induction on $k$.  If $k = 1$, this is immediate from the definition of prime.  Suppose that the result is true for a fixed $k \geq 1$.  Suppose that $p^{k+1} \mid ab$ and $p \nmid a$.  Since $k \geq 1$, we have $p \mid ab$, so since $p$ is prime we know that $p \mid b$.  Write $b = pc$ for some $c \in A$.  We then have $p^{k+1} \mid apc$, so $p^k \mid ac$.  By induction, we conclude that $p^k \mid c$.  Since $b = pc$, it follows that $p^{k+1} \mid b$.
\end{proof}

\begin{theorem} \label{t:PropertiesOfLocalization}
Let $A$ be a computable Noetherian UFD and let $q \in A$ be prime.  Suppose that $\{a \in A : q \mid a$ in $A\}$ is a computable set.  Let $S = \{1,q,q^2,\dots\}$ and let $B = S^{-1}A$ as above.  
\begin{enumerate}
\item We can build $B$ as a computable extension of $A$ uniformly from $A$ and an index for the set $\{a \in A : q \mid a$ in $A\}$ of multiples of $q$.
\item Let $p \in Primes(A) \backslash Associates_A(q)$.  The multiples of $p$ in $B$ are precisely the elements of the following set:
\[
\{a \in A : p \mid a \text{ in } A\} \cup \left\{\frac{a}{q^k} : a \in A, k \geq 1, q \nmid a \text{ in } A, \text{ and } p \mid a \text{ in } A\right\}
\]
In particular, there are no new elements of $A$ that are multiples of $p$ in $B$.  Furthermore, if we have a computable index for the set $\{a \in A : p \mid a$ in $A\}$, then we can uniformly computably obtain a computable index for the set $\{\sigma \in B : p \mid \sigma$ in $B\}$.
\item If $p_1,p_2 \in Primes(A)$ are not associates in $A$, then they are not associates in $B$.
\item $Primes(A) \backslash Associates_A(q) \subseteq Primes(B)$.
\item $B$ is a Noetherian UFD.
\end{enumerate}
\end{theorem}

\begin{proof}
\begin{enumerate}
\item Immediate from above.

\item It is easy to see that the elements in the given sets are indeed multiples of $p$ in $B$.  Suppose then that $\sigma \in B$ is arbitrary with $p \mid \sigma$ in $B$.  Suppose first that $\sigma = a \in A$.  We need to show that $p \mid a$ in $B$.  We have two cases.
\begin{itemize}
\item Suppose that there exists $b \in A$ with $pb = \sigma = a$.  We then trivially have that $p \mid a$ in $A$.
\item Suppose instead there exists $b \in A$ with $q \nmid b$ and $\ell \geq 1$ such that $p \cdot \frac{b}{q^{\ell}} = a$.  We then have $pb = aq^{\ell}$.  Thus $p \mid aq^{\ell}$ in $A$, and since $p$ is prime and $p \nmid q$ (as $p \notin Associates_A(q)$), it follows that $p \mid a$ in $A$.
\end{itemize}
Suppose instead that that $\sigma = \frac{a}{q^k}$ where $a \in A$ with $q \nmid a$ and $k \geq 1$.  We need to show that $p \mid a$ in $A$.
\begin{itemize}
\item Suppose that there exists $b \in A$ with $pb = \sigma = \frac{a}{q^k}$.  We then have $pbq^k = a$, so $p \mid a$ in $A$.
\item Suppose instead that there exists $b \in A$ with $q \nmid b$ and $\ell \geq 1$ such that $p \cdot \frac{b}{q^{\ell}} = \sigma = \frac{a}{q^k}$.  We then have $pbq^k = aq^{\ell}$.  Thus $p \mid aq^{\ell}$ in $A$, and since $p$ is prime and $p \nmid q$ (as $p \notin Associates_A(q)$), it follows that $p \mid a$ in $A$.
\end{itemize}
This completes the proof.

\item We prove the contrapositive.  Suppose that $p_1$ and $p_2$ are associates in $B$.  Fix $\sigma \in U(B)$ such that $p_1 = \sigma p_2$.  We know the units of $B$ from Proposition \ref{p:UnitsInLocalization}, so we handle the cases.
\begin{itemize}
\item If $\sigma \in U(A)$, then clearly $p_1$ and $p_2$ are associates in $A$.
\item Suppose that $\sigma = uq^k$ with $u \in U(A)$.  We then have $p_1 = uq^kp_2$, so $p_2 \mid p_1$ in $A$.  Since $p_1$ is prime in $A$, it is irreducible in $A$, so as $p_2$ is not a unit we can conclude that $p_1$ and $p_2$ are associates in $A$.
\item Suppose that $\sigma = \frac{u}{q^k}$ where $k \geq 1$ and $u \in U(A)$.  We then have $p_1 = \frac{u}{q^k} \cdot p_2$, so $p_1u^{-1}q^k = p_2$.  This implies that $p_1 \mid p_2$ in $A$.  As in the previous case, this implies that $p_1$ and $p_2$ are associates in $A$.
\end{itemize}

\item Let $p \in Primes(A) \backslash Associates_A(q)$.  First notice that $p \notin U(B)$ from Proposition \ref{p:UnitsInLocalization} because $p \notin U(A)$ and $p$ is not an associate of any $q^k$ (because if $p \mid q^k$, then $p \mid q$ as $p$ is prime, and hence $p$ is an associate of $q$).  Suppose that
\[
\frac{p}{1} \mid \frac{a}{q^k} \cdot \frac{b}{q^{\ell}}
\]
where we allow the possibility that $k = 0$ and/or $\ell = 0$.  Fix $c \in A$ and $m \geq 0$ with
\[
\frac{p}{1} \cdot \frac{c}{q^m} = \frac{a}{q^k} \cdot \frac{b}{q^{\ell}}
\]
We then have $pcq^{k+\ell} = q^mab$, so $p \mid q^mab$ in $A$.  Now $p$ is prime and $p \nmid q$ (as $p \notin Associates_A(q)$), so either $p \mid a$ in $A$ or $p \mid b$ in $A$.  If $p \mid a$ in $A$, then it is easy to see that $\frac{p}{1} \mid \frac{a}{q^k}$ in $B$.  Similarly if $p \mid b$.  Therefore, $p \in Prime(B)$.

\item This is immediate from the fact that the localization of a Noetherian ring is a Noetherian ring, and the localization of a UFD is a UFD.
\end{enumerate}
\end{proof}

Notice that using this machinery we can prove the result (essentially appearing \cite{Baur} and \cite[Example 4.3.9]{SHTucker}) that there exists a computable PID $A$ such that $U(A)$ is $\Sigma^0_1$-complete in all computable presentations.  Fix a $\Sigma_1^0$-complete set $Q$.  Start with $A_0 = \mathbb{Z}$ and let $p_0,p_1,p_2,\dots$ be a listing of the usual primes.  As we go along, if we have $A_n$ and we ever see that $e \in Q$, then we perform our unit construction to build $A_{n+1}$ extending $A_n$ so that $p_e \in U(A_{n+1})$ while maintaining primeness of the $p_i$ not equal to $p_e$ or to any elements we already made units.  Let $A = A_{\infty} = \bigcup_{n=0}^{\infty} A_n$ and notice that $i \in Q$ if and only if $p_i \in U(A_{\infty})$.  Since the final ring $A_{\infty}$ is a localization of the PID $A_0 = \mathbb{Z}$, it follows that $A_{\infty}$ is a PID.

\section{Introducing a Factorization}

In this section, we suppose that we have a computable Noetherian UFD $A$ and an element $q \in Primes(A)$.  We introduce a new factorization of $q$ by going to the ring $B = A[x,y] / \langle xy - q \rangle$.  The hope is that we only destroy the primeness/irreducibility of $q$ (and its associates), and we leave enough flexibility so that we can later make $y$ a unit without making $x$ a unit (so that then $q$ and $x$ will be associates).

\begin{proposition}
$B$ is an integral domain.
\end{proposition}

\begin{proof}
We claim that $xy - q$ is irreducible in $A[x,y]$.  Use the lexicographic monomial ordering in $A[x,y]$ with $x > y$.  Recall that, under this ordering, the multi-degree of $\sum_{k=1}^n c_kx^{i_k}y^{j_k} \in A[x,y]$ is the lexicographically largest of the pairs $(i_1,j_1),\dots,(i_n,j_n)$.  Notice that the multi-degree of $xy - q$ is $(1,1)$, so if it factors, the leading terms must have multi-degree $(1,1)$ and $(0,0)$ or $(1,0)$ and $(0,1)$ (since multi-degrees add upon multiplication).  The former implies that one of the factors is constant, and hencemust be a unit since the leading coefficient of $xy - q$ is $1$.  Consider the latter.  We have
\[
xy - q = (ax + by + c)(dy + e)
\]
where $a \neq 0$ and $d \neq 0$.  Comparing coefficients of $x$ on each side, we conclude that $ae = 0$, so $e = 0$ (because $A$ is an integral domain).  Comparing constants, we conclude that $q = -ce$, so $q = 0$, a contraction.

Since $xy - q$ is irreducible in $A[x,y]$ and $A[x,y]$ is a UFD (because $A$ is a UFD), we conclude that $xy - q$ is prime in $A[x,y]$.  Therefore, the quotient $B = A[x,y] / \langle xy - q \rangle$ is an integral domain.
\end{proof}

\begin{proposition} \label{p:RepresentingElementsInFactorization}
Every element of $B$ can be represented uniquely in the form
\[
a_mx^m + \dots + a_1x + c + b_1y + \dots + b_ny^n
\]
where each $a_i \in A$, $b_i \in A$, and $c \in A$. 
\end{proposition}

\begin{proof}
Given an arbitrary polynomial $h(x,y) \in A[x,y]$, we can divide by $xy - q$ (using the fact that the leading term is a unit) to obtain a remainder where no monomial is divisible by $xy$.  In other words, in the quotient, reduce any monomial with $xy$ in it to $q$, and repeat until there are no $xy$'s.  This proves existence.  For uniqueness, the difference of any polynomials of this form is another polynomial of this form, and hence has no monomial containing both an $x$ and a $y$.  Any nonzero multiple of $xy - q$ must have a monomial divisible by $xy$ by looking a leading term under some monomial ordering (and again using the fact that $A$ is an integral domain).
\end{proof}

Notice that in $B$ we have $xy = q$.  Thinking of $y = \frac{q}{x}$, we can alternatively think about $B$ in the following way.

\begin{proposition}
Consider the following subring of $A(x)$:
\[
A\left[x,\frac{q}{x}\right] =
\left\{
a_mx^m + \dots + a_1x + a_0 + a_{-1} \cdot \frac{q}{x} + \dots + a_{-n} \cdot \frac{q^n}{x^n} : a_i \in A
\right\}
\]
We have $B \cong A[x,\frac{q}{x}]$.
\end{proposition}

\begin{proof}
Define a ring homomorphism $\varphi \colon A[x,y] \to A(x)$ by fixing $A$ pointwise, sending $x \mapsto x$, and sending $y \mapsto \frac{q}{x}$.  We claim that $\ker(\varphi) = \langle xy - q \rangle$.  First notice that $xy - q \in \ker(\varphi)$, so we certainly have $\langle xy - q \rangle \subseteq \ker(\varphi)$.  Let $f(x,y) \in \ker(\varphi)$.  Divide by $xy - q$ to write $f(x,y) = (xy - q) \cdot g(x,y) + r(x,y)$ where $r(x,y)$ has no monomial having both an $x$ and a $y$.  Notice that $\varphi(r(x,y)) = \varphi(f(x,y)) = 0$.  Writing
\[
r(x,y) = a_mx^m + \dots + a_1x + c + b_1y + \dots + b_ny^n
\]
we then have that
\[
a_mx^m + \dots + a_1x + c + b_1 \cdot \frac{q}{x} + \dots + b_n \cdot \frac{q^n}{x^n} = 0
\]
Multiplying by $x^n$ and looking at coefficients, we conclude that each $a_i = 0$, each $b_i = 0$, and $c = 0$.  Thus $r(x,y) = 0$, and hence $f(x,y) \in \langle xy - q \rangle$.

Since $B = A[x,y]/\langle xy - q \rangle$, it follows that we get an induced injective homomorphism $\hat{\varphi} \colon B \to A(x)$.  Since every element of $B$ can be represented in the form $a_mx^m + \dots + a_1x + c + b_1y + \dots + b_ny^n$, we conclude that $\text{range}(\varphi) = \text{range}(\hat{\varphi}) = A[x,\frac{q}{x}]$, and hence $B \cong A[x,\frac{q}{x}]$.
\end{proof}

We will use the different ways of representing elements of the extension $B \cong A\left[x,\frac{q}{x}\right]$ interchangeably depending on which is most convenient.  With this isomorphism in mind, we define the following two functions.

\begin{definition}
Define $\deg_x \colon B \backslash \{0\} \to \mathbb{Z}$ as follows.  Let $\sigma \in B$ and consider the unique representation of $\sigma$ given in Proposition \ref{p:RepresentingElementsInFactorization}.
\begin{itemize}
\item If there is a term containing a power of $x$ with a nonzero coefficient, then $\deg_x(\sigma)$ is the largest such power of $x$.
\item If there is no such power of $x$, but there is a nonzero constant term, then $\deg_x(\sigma) = 0$.
\item If there is no such power of $x$ and no constant term, let $m$ be the least power of $y$ with a nonzero coefficient, and define $\deg_x(\sigma) = -m$.
\end{itemize}
We define $\deg_y \colon B \backslash \{0\} \to \mathbb{Z}$ similarly.
\end{definition}

For example, we have $\deg_x(y^2 + y^5) = -2$ and $\deg_y(y^2 + y^5) = 5$.

\begin{proposition}
Let $\sigma,\tau \in B \backslash \{0\}$.  We then have
\begin{align*}
\deg_x(\sigma\tau) & = \deg_x(\sigma) + \deg_x(\tau) \\
\deg_y(\sigma\tau) & = \deg_y(\sigma) + \deg_y(\tau)
\end{align*}
\end{proposition}

\begin{proof}
It is straightforward to prove this in the case when $\sigma$ and $\tau$ are monomials, i.e.~of the form $ax^k$, $by^{\ell}$, or $c \neq 0$ (notice that here we use the fact that $A$ is an integral domain to conclude that the product is a nonzero monomial).  For general $\sigma$ and $\tau$, we need only examine the leading $x$-terms or $y$-terms.
\end{proof}

\begin{proposition} \label{p:SumOfDegreesNonnegative}
We have
\[
\deg_x(\sigma) + \deg_y(\sigma) \geq 0
\]
for all $\sigma \in B \backslash \{0\}$, with equality if and only if $\sigma$ is a constant times a monomial.
\end{proposition}

\begin{proof}
Let $\sigma \in B \backslash \{0\}$. If the leading $x$-term is $x^m$, then $\deg_x(\sigma) = m$ and $\deg_y(\sigma) \geq -m$, with equality if and only if $x^m$ is the leading $y$-term as well.  A similar argument works if the leading $y$-term is $y^n$.  Otherwise, we only have a constant, it which case both $\deg_x(\sigma) = 0$ and $\deg_y(\sigma) = 0$.
\end{proof}

\begin{proposition} \label{p:ProductsLandingInA}
Let $\sigma,\tau \in B$.  We then have that $\sigma\tau \in A$ in exactly the following cases:
\begin{itemize}
\item $\sigma = 0$ or $\tau = 0$.
\item $\sigma \in A$ and $\tau \in A$.
\item There exist $a,b \in A$ and $n \in \mathbb{N}^+$ with $\sigma = ax^n$ and $\tau = by^n$, or there exists $a,b \in A$ and $n \in \mathbb{N}^+$ with $\sigma = by^n$ and $\tau = ax^n$.
\end{itemize}
\end{proposition}

\begin{proof}
In each of these cases it is easy to see that $\sigma\tau \in A$.  Suppose conversely that $\sigma\tau \in A$.  We may assume that $\sigma \neq 0$ and $\tau \neq 0$ or else we are done.  We then have
\[
\deg_x(\sigma) + \deg_x(\tau) = \deg_x(\sigma\tau) = 0
\]
so $\deg_x(\tau) = -\deg_x(\sigma)$.  Similarly, we have $\deg_y(\tau) = -\deg_y(\sigma)$.  Adding these gives
\[
\deg_x(\tau) + \deg_y(\tau) = -\deg_x(\sigma) - \deg_y(\sigma) = -(\deg_x(\sigma) + \deg_y(\sigma))
\]
Using Proposition \ref{p:SumOfDegreesNonnegative}, the only possibility is that $\deg_x(\tau) + \deg_y(\tau) = 0 = \deg_x(\sigma) + \deg_y(\sigma)$, and hence that both $\sigma$ and $\tau$ are constants times monomials.  The result now follows.
\end{proof}

\begin{corollary}
Let $a \in A$ with $q \nmid a$ in $A$.  If $\sigma \in B$ and $\sigma \mid a$ in $B$, then $\sigma \in A$ and $\sigma \mid a$ in $A$.  In other words, the set of divisors of $a$ in $B$ equals the set of divisors of $a$ in $A$.
\end{corollary}

\begin{proof}
By Proposition \ref{p:ProductsLandingInA}, the only possible new divisors of $a$ are when $a = bx^n \cdot cy^n$ with $n \geq 1$.  However, this implies that $a = bc \cdot q^n$, so $q \mid a$ in $A$.
\end{proof}

\begin{corollary} \label{c:UnitsAfterIntroducingFactorization}
The units of $B$ are precisely the units of $A$, i.e.~$U(B) = U(A)$.
\end{corollary}

\begin{proof}
This is immediate because the set of units is the set of divisors of $1 \in A$.
\end{proof}


\begin{theorem} \label{t:PropertiesOfFactorization}
Let $A$ be a computable Noetherian UFD and let $q \in A$ be prime.  Let $B = A[x,y] / \langle xy - q \rangle$ as above.
\begin{enumerate}
\item We can build $B$ as a computable extension of $A$ uniformly.
\item If $p_1,p_2 \in Primes(A)$ are not associates in $A$, then they are not associates in $B$.
\item \label{e:MultOfPInB} Let $p \in Primes(A) \backslash Associates_A(q)$ and let $\sigma \in B$.  We have that $p \mid \sigma$ in $B$ if and only every coefficient of $\sigma$ is divisible by $p$ in $A$.  In particular, there are no new elements of $A$ that are multiples of $p$ in $B$.  Furthermore, if we have a computable index for the set $\{a \in A : p \mid a$ in $A\}$, then we can uniformly computably obtain a computable index for the set $\{\sigma \in B : p \mid \sigma$ in $B\}$.
\item $Primes(A) \backslash Associates_A(q) \subseteq Primes(B)$.
\item \label{e:MultOfXInB} $x \mid \sigma$ in $B$ if and only if the constant term and the coefficients of each $y^j$ in $\sigma$ are all divisible by $q$ in $A$.  Therefore, if we have a computable index for the set $\{a \in A : q \mid a$ in $A\}$, then we can uniformly computably obtain a computable index for the set $\{\sigma \in B : x \mid \sigma$ in $B\}$.
\item $y \mid \sigma$ in $B$ if and only if the constant term and the coefficients of each $x^i$ in $\sigma$ are all divisible by $q$ in $A$.  Therefore, if we have a computable index for the set $\{a \in A : q \mid a$ in $A\}$, then we can uniformly computably obtain a computable index for the set $\{\sigma \in B :y \mid \sigma$ in $B\}$.
\item \label{e:XAndYArePrimesInB} $x$ and $y$ are primes in $B$ that are not associates of each other in $B$.
\item $x$ and $y$ are not associates in $B$ with any element of $A$, and hence not with any element of $Primes(A)$.
\item $B$ is a Noetherian UFD.
\end{enumerate}
\end{theorem}

\begin{proof}
\begin{enumerate}
\item Immediate from Proposition \ref{p:RepresentingElementsInFactorization}.

\item Immediate from Corollary \ref{c:UnitsAfterIntroducingFactorization}.

\item This follows from the fact that
\begin{align*}
p \cdot (a_mx^m & + \dots + a_1x + c + b_1y + \dots + b_ny^n) \\
& = pa_mx^m + \dots + pa_1x + pc + pb_1y + \dots + pb_ny^n.
\end{align*}

\item Let $p  \in Primes(A) \backslash Associates_A(q)$.  Notice that $p$ is nonzero and is not a unit of $B$ by Corollary \ref{c:UnitsAfterIntroducingFactorization}.  Let $\sigma,\tau \in B$ and suppose that $p \mid \sigma\tau$.  Assume that $p \nmid \sigma$ and $p \nmid \tau$.  We clearly have that both $\sigma$ and $\tau$ are nonzero.  Using \ref{e:MultOfPInB}, we know that $p$ divides every coefficient of $\sigma\tau$ in $A$, but there are coefficients of $\sigma$ and $\tau$ that are not divisible by $p$ in $A$.  Write
\[
\sigma = a_mx^m + \dots + a_1x + a_0 + a_{-1} \cdot \frac{q}{x} + \dots + a_{-n} \cdot \frac{q^n}{x^n}
\]
and
\[
\tau = b_mx^m + \dots + b_1x + b_0 + b_{-1} \cdot \frac{q}{x} + \dots + b_{-n} \cdot \frac{q^n}{x^n}
\]
Let $k$ and $\ell$ be largest possible such that $p \nmid a_k$ in $A$ and $p \nmid b_{\ell}$ in $A$.  Look at the coefficient of $x^{k+\ell}$ in $\sigma\tau$.  This coefficient will be a sum of terms, one of which is $a_kb_{\ell}q^j$ for some $j$, while other terms will be divisible by $p$ in $A$.  Since $p$ divides the resulting coefficient, it follows that $p \mid a_kb_{\ell}q^j$ in $A$.  However, this is a contradiction because $p$ is prime in $A$ but divides none of $a_k$, $b_{\ell}$, or $q$ (the last because $p$ is not an associate of $q$ in $A$).

\item Let $\sigma \in B$ and write
\[
\sigma = a_mx^m + \dots + a_1x + c + b_1y + \dots + b_ny^n
\]
Suppose first that $q \mid c$ and $q \mid b_j$ in $A$ for each $j$.  Fix $e \in A$ with $c = qe$ and fix $d_j \in A$ such that $b_j = qd_j$ for all $j$.  We then have
\[
x \cdot (a_mx^{m-1} + \dots + a_1 + ey + d_1y^2 + \dots + d_ny^{n+1}) = \sigma
\]
Conversely, suppose that $x \mid \sigma$, so that
\[
\sigma = x \cdot (a_mx^m + \dots + a_1x + c + b_1y + \dots + b_ny^n)
\]
for some $a_i,c,b_j \in A$.  Then we have
\[
\sigma = a_mx^{m+1} + \dots a_1x + cx + qb_1 + qb_2y + \dots qb_ny^{n-1}
\]

\item Similar to \ref{e:MultOfXInB}.

\item Notice that $x$ is nonzero and is not a unit of $B$ by Corollary \ref{c:UnitsAfterIntroducingFactorization}.  Let $\sigma,\tau \in B$ and suppose that $x \mid \sigma\tau$.  Assume that $x \nmid \sigma$ and $x \nmid \tau$.  We clearly have that both $\sigma$ and $\tau$ are nonzero.  Using \ref{e:MultOfXInB}, we know that $q$ divides the constant term and the coefficients of each $y^j$ in $\sigma\tau$ in $A$.  Write
\[
\sigma = a_mx^m + \dots + a_1x + a_0 + b_1y + \dots + b_ny^n
\]
and
\[
\tau = c_mx^m + \dots + c_1x + c_0 + d_1y + \dots + d_ny^n
\]
By \ref{e:MultOfXInB}, we may let $k$ and $\ell$ be largest possible such that $q \nmid a_k$ in $A$ and $q \nmid c_{\ell}$ in $A$.  Look at the coefficient of $x^{k+\ell}$ in $\sigma\tau$.  This coefficient will be a sum of terms, one of which is $a_kc_{\ell}$, while other terms will be divisible by $q$ in $A$.  Since $q$ divides the resulting coefficient, it follows that $q \mid a_kc_{\ell}$ in $A$.  However, this is a contradiction because $q$ is prime in $A$ but divides neither of $a_k$ or $c_{\ell}$ in $A$.

The proof that $y$ is prime in $B$ is similar.  The fact that $x$ and $y$ are not associates in $B$ follows from Corollary \ref{c:UnitsAfterIntroducingFactorization}.

\item Immediate from Corollary \ref{c:UnitsAfterIntroducingFactorization}.

\item We are assuming that $A$ is a Noetherian UFD.  Since $A$ is Noetherian, we know that $A[x,y]$ is Noetherian by the Hilbert Basis Theorem.  Since $B$ is a quotient of $A[x,y]$, it follows that $B$ is also Noetherian.  We also know from \ref{e:XAndYArePrimesInB} that $x$ is prime in $B \cong A[x,\frac{q}{x}]$.  To argue that $B$ is a UFD, we use Nagata's Criterion (see \cite[Theorem 20.2]{Matsumura} or \cite[Lemma 19.20]{Eisenbud}) which says the following.  

\begin{theorem}[Nagata's Criterion]
Let $B$ be a Noetherian integral domain.  Let $\Gamma$ be a set of prime elements of $B$, and let $S$ be the multiplicative set generated by $\Gamma$.  If $S^{-1}B$ is a UFD, then $B$ is a UFD.
\end{theorem}

Now $x$ is prime in $B \cong A[x,\frac{q}{x}]$ by \ref{e:XAndYArePrimesInB}.  The localization of $A[x,\frac{q}{x}]$ at $x$ equals $A[x,\frac{q}{x},\frac{1}{x}] = A[x,\frac{1}{x}]$, which is the localization of $A[x]$ at $x$.  Since $A$ is a UFD, we know that $A[x]$ is a UFD .  Since any localization of a UFD is a UFD, it follows that $A[x,\frac{1}{x}]$ is a UFD.  Since $B$ is a Noetherian integral domain, $x$ is prime in $B$, and $B$ localized at $x$ is a UFD, we may use Nagata's Criterion to conclude that $B$ is a UFD.
\end{enumerate}
\end{proof}

\section{Construction and Verification}

We now prove Theorem \ref{t:Pi2ControlOfPrimes}.  Let $Q$ be an arbitrary $\Pi^0_2$ set.  Fix a computable $R \subseteq \mathbb{N}^3$ such that
\[
i \in Q \Longleftrightarrow (\forall w)(\exists z) R(w,z,i)
\]
Fix a bijective computable pairing function $\langle \cdot, \cdot \rangle \colon \mathbb{N} \times \mathbb{N} \to \mathbb{N}$ with the property that $\langle i,s \rangle < \langle i,t \rangle$ whenever $s < t$.

We work in stages, and begin by initializing with $A_0 = \mathbb{Z}$.  We now start at stage $0$.  At a given stage, we will have introduced finitely many $x_i^{(k)}$ and $y_i^{(k)}$ for each $i$, and we will have marked a finite initial segment of $\mathbb{N}$ corresponding to those $w \in \mathbb{N}$ for which we have found a witnessing $z$ and done an action.  Furthermore, if $i$ has been initialized and the first unmarked $w$ is $k$, then we will have introduced $x_i^{(\ell)}$ and $y_i^{(\ell)}$ for each $\ell \leq k$, but we will not yet have introduced $x_i^{(k+1)}$ and $y_i^{(k+1)}$.

Suppose that we are now at a stage $\langle i,s \rangle$ and we have constructed through ring $A_n$ at this stage.
\begin{itemize}
\item If $s = 0$, we do an initialization for $p_i$ by introducing a first factorization.  In other words, we introduce $x_i^{(0)}$ and $y_i^{(0)}$ and perform the factorization construction on $p_i$ to create the ring $A_{n+1}$ (so we fill in one more column), and then move on to the next stage.
\item Suppose that $s \geq 1$, and let $k$ be the first unmarked $w$ corresponding to $i$.  Check to see if there exists $z \leq s$ such that $R(w,z,i)$.  If not, we do nothing and move to the next stage.  If so, we mark $k$ for $i$, and we act for $i$ at this stage, meaning that we do the following.  As mentioned above, we will have introduced through $x_i^{(k)}$ and $y_i^{(k)}$.  First, we perform the localization construction to make $y_i^{(k)}$ a unit in order to create the ring $A_{n+1}$.  Next, we introduce $x_i^{(k+1)}$ and $y_i^{(k+1)}$ and perform the factorization construction with these on $p_i$ to create the ring $A_{n+2}$.  Thus, we fill in two more columns in succession, and then move on to the next stage.
\end{itemize}
Finally, let $A_{\infty} = \bigcup_{n=0}^{\infty} A_n$.

\begin{theorem} \label{t:PrimesAndAssociatesAtStages}
Suppose that we are at the beginning of a given stage and we have constructed through $A_n$.  For each $i$ that has been initialized, let $x_i^{(k_i)}$ and $y_i^{(k_i)}$ be the last elements introduced for $i$ (so $k_i$ is the first unmarked $w$ for $i$).
\begin{itemize}
\item Suppose that $i$ has not yet been initialized.  We have the following:
\begin{itemize}
\item $p_i$ is prime in $A_n$.
\item The set $\{a \in A_n : p_i \mid a$ in $A_n\}$ is computable and we can uniformly find a computable index for it.
\item For any uninitialized $j \neq i$, we have that $p_i$ is not an associate of $p_j$ in $A_n$.
\item For any initialized $j \neq i$, we have that $p_i$ is not an associate of either $x_j^{(k_j)}$ or $y_j^{(k_j)}$ in $A_n$.
\end{itemize}
\item Suppose that $i$ has been initialized.  We have the following:
\begin{itemize}
\item $x_i^{(k_i)}$ and $y_i^{(k_i)}$ are prime in $A_n$, and are not associates in $A_n$.
\item The sets $\{a \in A_n : x_i^{(k_i)} \mid a$ in $A_n\}$ and $\{a \in A_n : y_i^{(k_i)} \mid a$ in $A_n\}$ are computable and we can uniformly find computable indices for them.
\item For any uninitialized $j \neq i$, we have that $x_i^{(k_i)}$ and $y_i^{(k_i)}$ are not associates of $p_j$ in $A_n$.
\item For any initialized $j \neq i$, we have that $x_i^{(k_i)}$ is not an associate of either $x_j^{(k_j)}$ or $y_j^{(k_j)}$ in $A_n$, and $y_i^{(k_i)}$ is not an associate of either $x_j^{(k_j)}$ or $y_j^{(k_j)}$ in $A_n$.
\end{itemize}
\item Suppose that we act for $i$ at this stage.  We then have $y_i^{(k_i)} \in U(A_{n+1})$, that $x_i^{(k_i)}$ is prime in $A_{n+1}$, and that $p_i$ is prime in $A_{n+1}$.
\end{itemize}
\end{theorem}

\begin{proof}
The proof is immediate by using induction on the stages along with Theorem \ref{t:PropertiesOfLocalization} and Theorem \ref{t:PropertiesOfFactorization}.
\end{proof}

\begin{definition}
Let $i,k \in \mathbb{N}$ and suppose that we introduce $x_i^{(k)}$ and $y_i^{(k)}$ in our construction.  We call $x_i^{(k)}$ and $y_i^{(k)}$ {\em terminal for $i$} if we never introduce $x_i^{(k+1)}$ and $y_i^{(k+1)}$ for $i$.
\end{definition}

\begin{proposition} \label{p:UntouchedPrimesRemainPrime}
We have the following.
\begin{enumerate}
\item Suppose that we introduce $x_i^{(k)}$ and $y_i^{(k)}$ in $A_m$.  If $x_i^{(k)}$ and $y_i^{(k)}$ are terminal for $i$, then they are non-associate primes in $A_n$ for each $n \geq m$.
\item If $r \in A_m$ is prime in $A_m$ and is not an associate of any $p_i$, $x_i^{(k)}$, or $y_i^{(k)}$ (whether terminal or nonterminal) in $A_m$, then $r$ remains prime in $A_n$ for each $n \geq m$.
\end{enumerate}
\end{proposition}

\begin{proof}
Again, this follows by induction using Theorem \ref{t:PropertiesOfLocalization} and Theorem \ref{t:PropertiesOfFactorization}.  
\end{proof}

\begin{proposition} \label{p:UnitsInLimit}
Let $a \in A_{\infty}$, so $a \in A_m$ for some $m \in \mathbb{N}$.  The following are equivalent:
\begin{enumerate}
\item $a \in U(A_{\infty})$.
\item $a \in U(A_n)$ for all sufficiently large $n \geq m$.
\item $a \in U(A_n)$ for some $n \geq m$.
\end{enumerate}
\end{proposition}

\begin{proof}
If $a \in U(A_{\infty})$, then fixing $b \in A_{\infty}$ with $ab = 1$, we have that $a \in U(A_n)$ for any $n$ large enough such that $a,b \in A_n$.  If $a \in U(A_n)$ for some $n \geq m$, then fixing $b \in A_n$ with $ab = 1$, we have $a,b \in A_{\infty}$, so $a \in U(A_{\infty})$.
\end{proof}

\begin{proposition} \label{p:PrimesInLimit}
Let $r \in A_{\infty}$, so $r \in A_m$ for some $m \in \mathbb{N}$.  If there are infinitely many $n \geq m$ such that $r$ is prime in $A_n$, then $r$ is prime in $A_{\infty}$.
\end{proposition}
	
\begin{proof}
Suppose that there are infinitely many $n \geq m$ such that $r$ is prime in $A_n$.  Fix $a,b \in A_{\infty}$ and suppose that $r \mid ab$ in $A_{\infty}$.  Fix $c \in A_{\infty}$ with $rc = ab$.  Go to a point where each of $r,c,a,b$ exist, and then fix an $n$ beyond that such that $r$ is prime in $A_n$.  We then have $r \mid ab$ in $A_n$, so as $r$ is prime in $A_n$, either $r \mid a$ in $A_n$ or $r \mid b$ in $A_n$.  Therefore, either $r \mid a$ in $A_{\infty}$ or $r \mid b$ in $A_{\infty}$.  Finally, notice that $r$ is nonzero and not a unit in $A_{\infty}$ because infinitely often it is not a unit in $A_n$ (as infinitely often it is prime in $A_n$).
\end{proof}

\begin{corollary} \label{c:FateOfXYAndUntouchedPrimes}
We have the following.
\begin{enumerate}
\item If $x_i^{(k)}$ and $y_i^{(k)}$ are introduced and are terminal for $i$, then they are non-associate primes in $A_{\infty}$.
\item If $x_i^{(k)}$ and $y_i^{(k)}$ are introduced and are nonterminal for $i$, then $y_i^{(k)} \in U(A_{\infty})$, and $x_i^{(k)}$ is an associate of $p_i$ in $A_{\infty}$.
\item If $r \in A_m$ is prime in $A_m$ and is not an associate of any $p_i$, $x_i^{(k)}$, or $y_i^{(k)}$ in $A_m$ (whether terminal or nonterminal), then $r$ remains prime in $A_{\infty}$.
\end{enumerate}
\end{corollary}

\begin{proof}
Immediate from Theorem \ref{t:PrimesAndAssociatesAtStages}, Proposition \ref{p:UntouchedPrimesRemainPrime}, Proposition \ref{p:UnitsInLimit}, and Proposition \ref{p:PrimesInLimit}.
\end{proof}

\begin{corollary} \label{c:PrimesCorrespondToQ}
$p_i$ is prime in $A_{\infty}$ if and only if $i \in Q$.
\end{corollary}

\begin{proof}
Suppose first that $i \in Q$.  We then act for $i$ infinitely often, and hence $p_i$ is prime in infinitely many $A_n$ by Theorem \ref{t:PrimesAndAssociatesAtStages}.  Thus, $p_i$ is prime in $A_{\infty}$ by Proposition \ref{p:PrimesInLimit}.

Suppose now that $i \notin Q$.  We then act for $i$ finitely often, so we may fix $k$ such that $x_i^{(k)}$ and $y_i^{(k)}$ are terminal for $i$.  By Corollary \ref{c:FateOfXYAndUntouchedPrimes}, each of $x_i^{(k)}$ and $y_i^{(k)}$ are prime in $A_{\infty}$.  Since $p_i = x_i^{(k)}y_i^{(k)}$, it follows that $p_i$ is not irreducible in $A_{\infty}$, and hence not prime in $A_{\infty}$.
\end{proof}

\begin{lemma} \label{l:FateOfPrimeAtFiniteStage}
Let $m \in \mathbb{N}$.  Let $r \in A_{\infty}$ and suppose that $r$ is prime in $A_m$.  We then have that either $r \in U(A_{\infty})$, $r$ is prime is $A_{\infty}$, or $r$ is the product of two primes in $A_{\infty}$.
\end{lemma}

\begin{proof}
We handle the various cases.
\begin{itemize}
\item If there exists $i \in Q$ such that $r$ is an associate of $p_i$ in $A_m$, then $r$ is prime in $A_{\infty}$ by Corollary \ref{c:PrimesCorrespondToQ}.
\item Suppose that there exists $i \notin Q$ such that $r$ is an associate of $p_i$ in $A_m$.  We then act for $i$ finitely often, so we may fix $k$ such that $x_i^{(k)}$ and $y_i^{(k)}$ are terminal for $i$.  By Corollary \ref{c:FateOfXYAndUntouchedPrimes}, each of $x_i^{(k)}$ and $y_i^{(k)}$ are prime in $A_{\infty}$.  We then have that $p_i = x_i^{(k)}y_i^{(k)}$, so $r = ux_i^{(k)}y_i^{(k)}$ for some unit $u \in U(A_{\infty})$.  Since $ux_i^{(k)}$ and $y_i^{(k)}$ are prime in $A_{\infty}$, we see that $r$ is the product of two primes in $A_{\infty}$.
\item If there exists $i,k \in \mathbb{N}$ such that $r$ is an associate of a terminal $x_i^{(k)}$ in $A_m$, then $r$ is prime in $A_{\infty}$ by Corollary \ref{c:FateOfXYAndUntouchedPrimes}.
\item If there exists $i,k \in \mathbb{N}$ such that $r$ is an associate of a terminal $y_i^{(k)}$ in $A_m$, then $r$ is prime in $A_{\infty}$ by Corollary \ref{c:FateOfXYAndUntouchedPrimes}.
\item If there exists $i \in Q$ and $k \in \mathbb{N}$ such that $r$ is an associate of a nonterminal $x_i^{(k)}$ in $A_m$, then $r$ is an associate of $p_i$ in $A_{\infty}$ by Corollary \ref{c:FateOfXYAndUntouchedPrimes} and hence $r$ is prime in $A_{\infty}$ by Corollary \ref{c:PrimesCorrespondToQ}.
\item If there exists $i \notin Q$ and $k \in \mathbb{N}$ such that $r$ is an associate of a nonterminal $x_i^{(k)}$ in $A_m$, then $r$ is an associate of $p_i$ in $A_{\infty}$ by Corollary \ref{c:FateOfXYAndUntouchedPrimes}, and hence $r$ is a product of two primes in $A_{\infty}$ from above.
\item If there exists $i,k \in \mathbb{N}$ such that $r$ is an associate of a nonterminal $y_i^{(k)}$ in $A_m$, then $r \in U(A_{\infty})$ by Corollary \ref{c:FateOfXYAndUntouchedPrimes}.
\item If $r$ is not an associate of any $p_i$, $x_i^{(k)}$, or $y_i^{(k)}$ in $A_m$, then $r$ is prime in $A_{\infty}$ by Corollary \ref{c:FateOfXYAndUntouchedPrimes}.
\end{itemize}
\end{proof}

\begin{theorem}
$A_{\infty}$ is a UFD.
\end{theorem}

\begin{proof}
We prove that every nonzero nonunit element of $A_{\infty}$ is a product of irreducibles and that every irreducible is prime, which suffices by Theorem \ref{t:EquivCharOfUFD}.

Let $a \in A_{\infty}$ be nonzero and not a unit.  Fix $n$ with $a \in A_n$ and notice that $a$ is not a unit in $A_n$.  Since $A_n$ is a UFD, we may write $a = r_1r_2 \cdots r_{\ell}$ where each $r_i$ is irreducible and hence prime in $A_n$.  By Lemma \ref{l:FateOfPrimeAtFiniteStage}, each $r_j$ is either a unit in $A_{\infty}$, is prime in $A_{\infty}$, or is the product of two primes in $A_{\infty}$.  It's not possible that all $r_j$ are units in $A_{\infty}$, because this would imply that $a$ is a unit in $A_{\infty}$.  Thus, $a$ is a product of primes in $A_{\infty}$ (since we can absorb the units in one of the primes).  Since primes are irreducible, we conclude that $a$ is a product of irreducibles in $A_{\infty}$.

We now show that every irreducible element of $A_{\infty}$ is prime.  Let $a \in A_{\infty}$ be irreducible.  Fix $n$ with $a \in A_n$.  Notice that $a$ is nonzero and not a unit in $A_n$ because otherwise it would be zero or a unit in $A_{\infty}$.  Since $A_n$ is a UFD, we may write $a = r_1r_2 \cdots r_{\ell}$ where each $r_j$ is irreducible and hence prime in $A_n$.  By Lemma \ref{l:FateOfPrimeAtFiniteStage}, each $r_j$ is either a unit in $A_{\infty}$, is prime in $A_{\infty}$, or is the product of two primes in $A_{\infty}$.  It's not possible that all $r_j$ are units in $A_{\infty}$, because this would imply that $a$ is a unit in $A_{\infty}$.  If some $r_j$ is a product of two primes in $A_{\infty}$, then $a$ is not irreducible in $A_{\infty}$, a contradiction.  Also, if two of the $r_j$ are prime in $A_{\infty}$, then $A$ is not irreducible in $A_{\infty}$, a contradiction.  Thus, exactly one of the $r_i$ is prime in $A_{\infty}$ and the rest are units.  It follows that $A$ is a prime times some units in $A_{\infty}$, so $a$ is prime in $A_{\infty}$.
\end{proof}

This completes the proof of Theorem \ref{t:Pi2ControlOfPrimes}.

\end{document}